\documentclass[12pt]{amsart}

\usepackage{latexsym,amsmath,amssymb,amsthm,amsfonts}

\usepackage[shortalphabetic]{amsrefs}

\newcommand{\R}{{\mathbb{R}}}

\newcommand{\E}{{\mathcal{E}}}
\renewcommand{\P}{{\mathcal{P}}}

\newcommand{\x}{{\boldsymbol{x}}}

\renewcommand{\u}{{\boldsymbol{u}}}
\renewcommand{\v}{{\boldsymbol{v}}}

\newcommand{\y}{{\boldsymbol{y}}}

\newcommand{\wt}{\widetilde}

\newtheorem{theorem}{Theorem}

\theoremstyle{remark}

\renewcommand{\phi}{\varphi}

\title{Pointwise behavior of Christoffel function on planar convex domains}

\author{A.\ Prymak}
\address{Department of Mathematics, University of Manitoba, Winnipeg, MB, R3T2N2, Canada}
\email{prymak@gmail.com}
\thanks{The first author was supported by NSERC of Canada Discovery Grant RGPIN 04863-15.}

\author{O.\ Usoltseva}
\address{Department of Mathematics, University of Manitoba, Winnipeg, MB, R3T2N2, Canada}
\email{usoltseo@myumanitoba.ca}
\thanks{The second author was supported by the University of Manitoba Graduate Fellowship and by the Department of Mathematics of the University of Manitoba.}



\begin{document}

\begin{abstract}
	We prove a general lower bound on Christoffel function on planar convex domains in terms of a modification of the parallel section function of the domain. For a certain class of planar convex domains, in combination with a recent general upper bound, this allows to compute the pointwise behavior of Christoffel function. We illustrate this approach for the domains $\{(x,y):|x|^\alpha+|y|^\alpha\le1\}$, $1<\alpha<2$, and compute up to a constant factor the required modification of the parallel section function, and, consequently, Christoffel function at an arbitrary interior point of the domain.
\end{abstract}

\maketitle

\section{Introduction and results}

Christoffel function associated with a compact set $D\subset\R^d$ having non-empty interior and with a positive weight function $w\in L_1(D)$ can be defined as
\begin{equation}\label{eqn:classic def}
\lambda_n(D,w,\x)=\left(\sum_{k=1}^N\phi_k(\x)^2\right)^{-1}, \quad \x\in D,
\end{equation}
where $\P_n=\P_{n,d}$ denotes the space of all real algebraic polynomials of total degree $\le n$ in $d$ variables, and $\{\phi_k\}_{k=1}^N$ is any orthonormal basis of $\P_n$ with respect to the inner product $\langle f,g \rangle = \int_D f(\y)g(\y)w(\y)d\y$. Christoffel function possesses the following well-known extremal property:
\begin{equation}\label{def_lambda}
\lambda_n(D,w,\x)=\min_{f\in\P_n,\, f(\x)=1}\int_D f^2(\y)w(\y) d\y, \quad \x\in D.
\end{equation}
For the uniform weight $w\equiv 1$, we set $\lambda_n(D,\x)=\lambda_n(D,w,\x)$.

Christoffel function is a valuable tool in various areas of analysis and mathematics. A common approach to the computation of Christoffel function is to use~\eqref{eqn:classic def} if an orthonormal basis of $\P_n$ is available, see, e.g.~\cite{Xu}. This becomes infeasible when $D$ is a rather general multivariate domain and the structure of the orthogonal polynomials on $D$ is unknown. A different approach is to use~\eqref{def_lambda} and comparison with other domains for which the behavior of Christoffel function is known, see, e.g.~\cite{Kr}, \cite{Di-Pr} and~\cite{Pr}. In this note we further develop this approach and focus on lower bounds on Christoffel function for planar convex domains and techniques for computation of the related geometric characteristics of the domain.

In what follows, $\partial D$ denotes the boundary of $D$. The constants $c$, $c(\cdot)$ are positive and depend only on parameters indicated in the parentheses (if any) and may be different at different occurrences even if the same notation is used.  The equivalence ``$\approx$'' is understood with absolute constants, namely, $A\approx B$ means $c^{-1}A\le B\le cA$.

Our main result is the following theorem.
\begin{theorem}\label{thm:main-lower}
	 Suppose $D\subset\R^2$ is a convex compact set with non-empty interior, $\x\in D\setminus\partial D$, $\u\in\R^2$ is a unit vector, $\beta$ and $\sigma$ are some positive constants. Let $\delta =\max\{q:\x+q\u\in D\}$ and
	\begin{equation}\label{eqn:def_li}
	l_i(t):= l_i(D,\x,t):=\max\{s:\x+(\delta-t)\u+(-1)^is\v\in D\}, \quad i=1,2, \quad 0<t<\beta,
	\end{equation}
	where $\v$ is one of the two unit vectors orthogonal to $\u$. If $\sigma n^{-2}<\delta<\beta/2$, then
	\begin{equation}\label{eqn:estimate-lower}
	\lambda_n(D,\x)\ge c(\beta,\sigma) n^{-2} \sqrt{\delta}\min_{i=1,2}\min_{\delta/2\le t \le \beta}\frac{l_i(t)}{\sqrt{t}}.
	\end{equation}
\end{theorem}

Remark that for the points very close to the boundary (within $\sigma n^{-2}$), the problem can be reduced to the case when $\delta>\sigma n^{-2}$ using~\cite[Proposition~1.4]{Pr}.

One can think of $l_i(t)$ as a modified parallel section function of $D$, and note that $l_1(t)+l_2(t)$ is the total length of the section of $D$ parallel to $\v$ through the point $\x+(\delta-t)\u$. Informally, the quantity on the right hand side of~\eqref{eqn:estimate-lower} describes the size of an ellipse that can be inscribed into $D$ so that $\x$ is sufficiently inside the ellipse. 

For the upper bound, by~\cite[Theorem~1.1]{Pr}, we have 
\begin{equation}\label{eqn:estimate-upper}
\lambda_n(D,\x) \le c(D,\sigma) n^{-2} \sqrt{\min\{ l_1(\delta) l_2(\delta), \delta \}}.
\end{equation}
(One can refer to~\cite{Pr} for specific geometric measurements of $D$ that affect the constant $c(D,\sigma)$ and are omitted here for simplicity.) Therefore, if
\begin{equation}\label{eqn:conditions}
l_1(t_1)\approx c(D) l_2(t_1)
\quad \text{and} \quad
\frac{l_i(t_1)}{\sqrt{t_1}}< c(D) \frac{l_i(t_2)}{\sqrt{t_2}},
 \quad \delta/2\le t_1\le t_2\le \beta, \quad i=1,2,
\end{equation}
then the estimates~\eqref{eqn:estimate-lower} and~\eqref{eqn:estimate-upper} match and we get \[\lambda_n(D,\x)\approx c(D,\beta,\sigma) n^{-2} l_1(\delta).\] As $l_1(\delta)\approx c(D)(l_1(\delta)+l_2(\delta))$, we see that the length of section of $D$ parallel to $\v$ through $\x$ is responsible for the magnitude of Christoffel function at $\x$, provided~\eqref{eqn:conditions} is satisfied. We remark that a natural choice for $\u$ would be the direction in which the distance from $\x$ to $\partial D$ is attained, although other choices are possible depending on specific situation. We believe that the class of convex bodies satisfying~\eqref{eqn:conditions} for some choice of $\u$ is rather wide. For this class, the combination of Theorem~\ref{thm:main-lower} and~\cite[Theorem~1.1]{Pr} provides geometric characterization of the behavior of Christoffel function at any point of the domain. We also note that for the upper estimate~\eqref{eqn:estimate-upper}, significantly fewer geometric measurements are needed (only $\delta$, $l_1(\delta)$ and $l_2(\delta)$) compared with~\eqref{eqn:estimate-lower} which requires the knowledge of $l_i(t)$ for $\delta/2\le t\le \beta$.

Next, we illustrate our main result for the domains $B_\alpha:=\{(x,y):|x|^\alpha+|y|^\alpha\le1\}$, $1<\alpha<2$. In particular, we show that these domains belong to the class of convex bodies described in the previous paragraph. To this end, for each interior point $\x$ within a constant distance from the boundary of the domain, we compute $l_i(D,\x,t)$ (see~\eqref{eqn:def_li}) explicitly up to a constant factor in terms of $t$ and $(x_0,y_0)$, a nearest point from the boundary to $\x$, i.e., $(x_0,y_0)$ is such that $|\x-(x_0,y_0)|=\min\{|\x-(x,y)|:(x,y)\in\partial B_\alpha\}$, where $|\cdot|$ is the Euclidean norm in $\R^2$. We note that, generally speaking, to find $l_i$ one needs to solve a non-linear equation. We hope that the techniques developed below to estimate $l_i$ for $B_\alpha$, which mostly result in equations of degree at most $2$, may prove useful for other planar convex bodies.

\begin{theorem}\label{cor:b-alpha-behaviour}
	Let $(x_0,y_0)\in\partial B_\alpha$, $1<\alpha<2$, $0\le x_0\le y_0$, $\u$ be the outward unit normal at $(x_0,y_0)$. There exists a constant $c_0(\alpha)>0$ depending only on $\alpha$ such that for 
	\begin{equation}\label{eqn:li-def-th2}
	l_i(t):= \max\{s:(x_0,y_0)-t\u+(-1)^is\v\in D\}, \quad i=1,2, \quad 0<t<1,
	\end{equation}
	where $\v$ is one of the two unit vectors orthogonal to $\u$, we have
	\begin{equation}\label{eqn:li-comp}
	l_i(t)\approx c(\alpha) t^{\frac12} (\max\{t, x_0^\alpha\})^{\frac1\alpha-\frac12}, \quad 0<t\le c_0(\alpha), \quad i=1,2.
	\end{equation}
	Further, if $\x\in B_\alpha\setminus\partial B_\alpha$ is such that $\delta:=|\x-(x_0,y_0)|=\min\{|\x-(x,y)|:(x,y)\in\partial B_\alpha\}$ and $\sigma n^{-2}\le \delta\le 1$, $\sigma>0$, then
	\begin{equation}\label{eqn:b-alpha}
	\lambda_n(B_\alpha,\x) \approx c(\alpha,\sigma) n^{-2} \delta^{\frac12} (\max\{\delta, x_0^\alpha\})^{\frac1\alpha-\frac12}.
	\end{equation}
\end{theorem}

The behavior of $\lambda_n(B_\alpha,\x)$ on $x=0$ and $y=0$ which contain the ``least smooth'' points $(0,\pm1)$ and $(\pm1,0)$ of $\partial B_\alpha$ has been studied in~\cite{Kr} and~\cite{Pr} and was essentially shown to be $n^{-2}\delta^{\frac1\alpha}$. In contrast, along $x=\pm y$, where the boundary is $C^2$ smooth, the behavior is $n^{-2}\delta^{\frac12}$, see~\cite[Proposition~3.3]{Pr}.    
Theorem~\ref{cor:b-alpha-behaviour} fills this gap by computing Christoffel function everywhere inside $B_\alpha$ and specifies how exactly the transition between different smoothness affects the behavior of Christoffel function. Also, Theorem~\ref{cor:b-alpha-behaviour} gives an affirmative answer to~\cite[Conjecture~3.4]{Pr}, moreover, the theorem provides the right hand side of~\cite[(3.4)]{Pr} up to a constant factor. 

While the results in this note are obtained for the uniform weight, they imply asymptotics of Christoffel function for other classes of weights using universality in the bulk~\cite{Kr-Lu}. 

\section{Proofs}

We begin with some preliminaries. By~\eqref{def_lambda},
\begin{equation}
\label{eqn:compare}
\text{if }D\subset \wt D\subset \R^2,\text{ then }
\lambda_n(D,\x)\le \lambda_n(\wt D,\x), \quad \x\in \wt D,
\end{equation}
and
\begin{equation}
\label{eqn:affine}
\lambda_n(TD,T\x)=\lambda_n(D,\x)|\det T|, \quad \x\in D,
\end{equation}
where $T\x=\x_0+A\x$ is any non-degenerate affine transform of $\R^2$, i.e., $\x_0\in\R^2$ and $A$ is a $2\times 2$ matrix, $\det T:=\det A\ne0$.

Let $B_2:=\{\x:|\x|\le1 \}$ denote the unit ball in $\R^2$. For $\sigma>0$, by~\cite[(2.3)]{Pr},
\begin{equation}\label{eqn:ball}
\lambda_n(B_2,(x,y))\approx c(\sigma) n^{-2} \sqrt{1-x^2-y^2}, \quad (x,y)\in (1-\sigma n^{-2})B_2.
\end{equation}

\begin{proof}[Proof of Theorem~\ref{thm:main-lower}]
	Denote
	\[
	\Lambda:= \sqrt{\frac{\beta}{6}} \min_{i=1,2}\min_{\delta/2\le t \le \beta}\frac{l_i(t)}{\sqrt{t}}.
	\]
	Consider the ellipse
	\[
	\E:=\left\{\x-t\u+s\v: \left(\frac{\frac\beta3+\frac\delta2-t}{\frac\beta3}\right)^2+ \left(   \frac{s}{\Lambda} \right)^2 \le 1 \right\}.
	\]
	If $t$ and $s$ satisfy the inequality from the definition of $\E$, then $\frac{\delta}{2}\le t<\frac{11}{12}\beta$ and
	\begin{align*}
	|s| &\le \Lambda \sqrt{1- \left(\frac{\frac\beta3+\frac\delta2-t}{\frac\beta3}\right)^2} 
	= \frac{3\Lambda}{\beta} \sqrt{2\tfrac{\beta}{3}\left(t-\tfrac\delta2\right)- \left(t-\tfrac\delta2\right)^2 } \\ 
	&\le \frac{3\Lambda}{\beta} \sqrt{2\tfrac{\beta}{3}\left(t-\tfrac\delta2\right)}
	\le \Lambda \sqrt{\frac{6}{\beta}} \sqrt{t} \le \min_{i=1,2} l_i(t),
	\end{align*} 
	so $\E\subset D$. Note that for an affine transform $T$ such that $T\E=B_2$ we have $\det T=\frac{3}{\Lambda\beta}$. Now by~\eqref{eqn:compare}, \eqref{eqn:affine} and~\eqref{eqn:ball},
	\[
	\lambda_n(D,\x)\ge \lambda_n(\E,\x)=\frac{\Lambda\beta}{3} \lambda_n(B_2,T\x) \approx \Lambda\beta c(\sigma) n^{-2}\sqrt{\delta},
	\] 
	implying~\eqref{eqn:estimate-lower}.
\end{proof}

\begin{proof}[Proof of Theorem~\ref{cor:b-alpha-behaviour}.]
	In addition to already set notations regarding constants, throughout this proof we use $c_j(\alpha)$, $j\ge0$, to denote different specific positive constants depending on $\alpha$ only. Note that $c(\alpha)$ may have different values at different occurrences, while for a fixed $j$ the value of $c_j(\alpha)$ may not. We emphasize that all the constants below do not depend on $x_0$ or $t$. We can assume that $v_x>0$ in $\v=(v_x,v_y)$ from~\eqref{eqn:li-def-th2}.

First we will show how~\eqref{eqn:li-comp} implies~\eqref{eqn:b-alpha}. Assuming~\eqref{eqn:li-comp}, we can apply Theorem~\ref{thm:main-lower} with $\beta=c_0(\alpha)$ and obtain the lower bound in~\eqref{eqn:b-alpha} if $\delta\le c_0(\alpha)/2$. If $\delta>c_0(\alpha)/2$, we note that $\delta B_2+\x\subset B_\alpha$, so by~\eqref{eqn:compare}, \eqref{eqn:affine} and~\eqref{eqn:ball}, 
\[
\lambda_n(B_\alpha,\x)\ge \lambda_n(\delta B_2+\x,\x)=\delta^2 \lambda_n(B_2,(0,0))\ge c(\alpha) n^{-2},
\]	
which proves the lower bound in~\eqref{eqn:b-alpha}. The upper bound in~\eqref{eqn:b-alpha} readily follows from~\eqref{eqn:estimate-upper}.

It remains to prove~\eqref{eqn:li-comp}. We remark that one can establish~\eqref{eqn:li-comp} for a wider range of $t$, e.g. for $0<t<\frac54$, but this requires some additional technicalities and is not needed for~\eqref{eqn:b-alpha}, which was our main goal.

We will select $c_0(\alpha)$ in the end of the proof. Now fix $t$ with $0<t\le c_0(\alpha)$ and set $(x_1,y_1)=(x_0,y_0)-t\u$. We assume that $x_0>0$, the case $x_0=0$ will be considered later.
		
	Suppose $y=l(x)$ is the equation of the line $\{(x_0,y_0)-{t} \u+(-1)^is\v:s\in\R\}$. Then
\[
l(x)=f(x_0)+f'(x_0)(x-x_0)-{t}\sqrt{1+(f'(x_0))^2}, 
\]
where $f(x)=(1-|x|^\alpha)^{\frac1\alpha}$ describes the upper half of $\partial B_\alpha$. We have $x_1=x_0+\frac{{t} f'(x_0)}{\sqrt{1+(f'(x_0))^2}}$.	Since $y_0=f(x_0)\ge x_0$, we obtain $x_0\le 2^{-\frac1\alpha}\le y_0$. For $0<x<1$, we have $f'(x)=-x^{\alpha-1}(1-x^\alpha)^{\frac1\alpha-1}$, and for $0<x<2^{-\frac1\alpha}$, we get 
$1<(1-x^\alpha)^{\frac1\alpha-1}<\sqrt{2}$. 
So,
\begin{equation}\label{eqn:f'(x_0)-bounds}
x_0^{\alpha-1}< -f'(x_0)< \sqrt{2}x_0^{\alpha-1}\quad\text{and}\quad
1<\sqrt{1+(f'(x_0))^2}< \sqrt{3}. 
\end{equation}
Now we can verify that $l(\pm1)>0=f(\pm1)$ provided ${t}<\frac{2^{-\frac1\alpha}-2^{-\frac12}}{\sqrt{3}}$, so we will require that $c_0(\alpha)<\frac{2^{-\frac1\alpha}-2^{-\frac12}}{\sqrt{3}}$. Hence, letting $x_2<x_3$ be the $x$-coordinates of the points of intersection of the line $y=l(x)$ with $\partial B_\alpha$, we obtain that $l(x_j)=f(x_j)$ and $l_{j-1}({t})=\sqrt{1+(f'(x_0))^2} (-1)^{j-1}(x_j-x_1)$, $j=2,3$. Therefore, due to~\eqref{eqn:f'(x_0)-bounds}, we need to show  $|x_j-x_1|\approx c(\alpha) {t}^{\frac12} (\max\{{t}, x_0^\alpha\})^{\frac1\alpha-\frac12}$, $j=2,3$.

We note that~\eqref{eqn:f'(x_0)-bounds} implies
\begin{equation}\label{eqn:x1x0-bound}
0\le x_0-x_1\le \sqrt{2} {t} x_0^{\alpha-1}.
\end{equation}

We define tangent parabolas to $y=f(x)$ at $x=x_0$ with varying quadratic term as follows:
\[
P(m,x):=f(x_0)+f'(x_0)(x-x_0)+\frac{m}{2}(x-x_0)^2.
\]
Note that for $m<0$ the equation $l(x)=P(m,x)$ has two solutions 
\begin{equation}\label{eqn:solutions}
x=x_0\pm \sqrt{\frac{2{t}\sqrt{1+(f'(x_0))^2}}{-m}}.
\end{equation}
Further, for any interval $[a,b]\subset [0,1]$ containing $x_0$, we have
\begin{equation}\label{eqn:parabolas}	
P(\min\{f''(t):t\in[a,b]\},x)\le f(x)\le P(\max\{f''(t):t\in[a,b]\},x), \quad x\in[a,b].
\end{equation}	

It is straightforward to compute that
\[
f'''(x)=-(\alpha-1)x^{\alpha-3}(1-x^\alpha)^{\frac1\alpha-3}((\alpha-2)+(\alpha+1)x^\alpha),
\]
so $f'''(x)>0$ for $x\in(0,c_1(\alpha))$, where $c_1(\alpha):=(\frac{2-\alpha}{1+\alpha})^{\frac1\alpha}$.

Now we show that
\begin{equation}\label{eqn:l1-small-delta}
\text{if }{t}\le c_2(\alpha) x_0^\alpha\text{ and } x_0\le c_1(\alpha)/2,\text{ then } |x_j-x_1|\approx c(\alpha) {t}^{\frac12} x_0^{1-\frac\alpha2}, \quad j=2,3,
\end{equation}
where $c_2(\alpha)$ will be selected later. By~\eqref{eqn:parabolas},
\begin{equation}\label{eqn:par-above-below-l1}
P(f''(x_0/2),x) \le f(x)\le P(f''(2x_0),x), \quad x\in [x_0/2,2x_0].
\end{equation}
Let $z_1<z_2$ and $z_3<z_4$ be the solutions of the quadratic equations $l(x)=P(f''(x_0/2),x)$ and $l(x)=P(f''(2x_0),x)$, respectively. Since $-f''(2^{\pm1}x_0)\approx c(\alpha) x_0^{\alpha-2}$, by~\eqref{eqn:solutions} we see that $c_3(\alpha){t}^\frac12 x_0^{1-\frac\alpha2} \le |z_j-x_0|\le c_4(\alpha) {t}^\frac12 x_0^{1-\frac\alpha2}$, $j=1,2,3,4$, for some positive constants $c_3(\alpha)$ and $c_4(\alpha)$ independent of the forthcoming choice of $c_2(\alpha)$. As ${t}^{\frac12}x_0^{1-\frac{\alpha}{2}} \le \sqrt{c_2(\alpha)} x_0$, if we impose that $c_2(\alpha)<(2c_4(\alpha))^{-2}$, then $x_0/2<z_3$ and $z_4<2x_0$. Now~\eqref{eqn:par-above-below-l1} implies that $z_3<x_2<z_1$ and $z_2<x_3<z_4$, so 
$c_3(\alpha){t}^\frac12 x_0^{1-\frac\alpha2} \le |x_j-x_0|\le c_4(\alpha) {t}^\frac12 x_0^{1-\frac\alpha2}$, $j=2,3$. If $c_2(\alpha)<(c_3(\alpha))^2/8$, then since ${t} x_0^{\alpha-1}\le \sqrt{c_2(\alpha)} {t}^{\frac12} x_0^{1-\frac\alpha2}$ we can use~\eqref{eqn:x1x0-bound} to see that $0\le x_0-x_1\le \frac{c_3(\alpha)}{2} {t}^\frac12 x_0^{1-\frac\alpha2}$ and conclude that~\eqref{eqn:l1-small-delta} holds provided $c_2(\alpha)$ is sufficiently small. Namely, we choose arbitrary $c_2(\alpha)>0$ satisfying $c_2(\alpha)<\min\{(2c_4(\alpha))^{-2},(c_3(\alpha))^2/8\}$.

Next we claim that
	\begin{equation}\label{eqn:C2-case}
	\text{if }x_0\in[c_1(\alpha)/2,2^{-\frac1\alpha}],\text{ then }|x_j-x_1|\approx c(\alpha){t}^{\frac12}, \quad j=2,3.
	\end{equation}
	The proof is similar to that of~\eqref{eqn:l1-small-delta} with certain differences as we will now outline. The interval $[x_0/2,2x_0]$ is replaced with $[c_1(\alpha)/3,1/2]$ and then we use that $c_5(\alpha)\le -f''(x)\le c_6(\alpha)$ for $x\in[c_1(\alpha)/3,1/2]$ and some positive constants $c_5(\alpha)$ and $c_6(\alpha)$, so that
	\[
	P(-c_6(\alpha),x) \le f(x) \le P(-c_5(\alpha),x), \quad x\in[c_1(\alpha)/3,1/2].
	\]
	Further, instead of requiring that $c_2(\alpha)$ is sufficiently small as was done for~\eqref{eqn:l1-small-delta}, we will require $c_0(\alpha)$ (and, hence, $t$) not to exceed a specific constant depending on $\alpha$ only, chosen to ensure that the analogs of $z_3$ and $z_4$ belong to $[c_1(\alpha)/3,1/2]$, and that $x_1-x_0$ does not exceed a sufficiently small constant times $t^{\frac12}$. We omit the details.

The proofs of the remaining estimates are different from the proofs of~\eqref{eqn:l1-small-delta} and~\eqref{eqn:C2-case} as we will mostly compare $f$ with lines rather than with parabolas. Define ${\tilde x}>0$ to be the point where $f({\tilde x})=l(x_0)$. It is straightforward that
\begin{equation}
\label{eqn:f'(x_0)-behav}
\frac{x^\alpha}{\alpha} \le 1-f(x)\le 2^{1-\frac1\alpha} \frac{x^\alpha}\alpha, \quad 0<x<2^{-\frac1\alpha}.
\end{equation}
This implies $1-f(x_0)\approx c(\alpha) x_0^\alpha$. By~\eqref{eqn:f'(x_0)-bounds}, we have $f(x_0)-l(x_0)\approx {t}$. So, if ${t}\ge c_2(\alpha) x_0^\alpha$, then $1-l(x_0)\approx c(\alpha){t}$ and due to~\eqref{eqn:f'(x_0)-behav} we obtain the following:
\begin{equation}\label{eqn:x10}
	\text{if }{t}\ge c_2(\alpha) x_0^\alpha,\text{ then }{\tilde x}\approx c(\alpha){t}^{\frac1\alpha}.
\end{equation} 

	Next we establish that
	\begin{equation}\label{eqn:l1-large-delta}
	\text{if }{t}\ge c_2(\alpha) x_0^\alpha\text{ and } x_0\le c_1(\alpha)/2,\text{ then } x_3-x_1\approx c(\alpha) {t}^{\frac1\alpha}.
	\end{equation}		 	
	Let $f^{-1}$ be the inverse of $f$ on $[0,1]$. Clearly, $f^{-1}$ is concave. Therefore, $f^{-1}(y)-f^{-1}(y+h)$ is decreasing in $y$ for fixed $h>0$. Applying this with $h=f(x_0)-l(x_0)$, we see that ${\tilde x}-x_0=f^{-1}(l(x_0))-f^{-1}(f(x_0))>f^{-1}(1-h)-f^{-1}(1)\approx c(\alpha) h^{\frac1\alpha}\approx c(\alpha){t}^{\frac1\alpha}$. Since $x_3>{\tilde x}$, by~\eqref{eqn:x1x0-bound} we have $x_3-x_1>{\tilde x}-x_0$, which yields the lower bound on $x_3-x_1$ in~\eqref{eqn:l1-large-delta}. 
	
	To estimate $x_3$ from above, we consider $L(x)$, the tangent line to $f$ at ${\tilde x}$, which, by concavity of $f$, satisfies $f(x)\le L(x)$, $x\in[0,1]$, and has the slope smaller than the slope of $l$. Therefore, letting ${\bar x}$ be the point of intersection of $l$ and $L$, we have the bound $x_3<{\bar x}$ and compute that
	\begin{equation}\label{eqn:x11}
	{\bar x}=\frac{f({\tilde x})-f(x_0)+{t}\sqrt{1+(f'(x_0))^2}+x_0f'(x_0)-{\tilde x}f'({\tilde x})}{f'(x_0)-f'({\tilde x})}.
	\end{equation}
	Due to~\eqref{eqn:x10} and ${\tilde x}-x_0\approx c(\alpha){t}^{\frac1\alpha}$,  we estimate $f'(x_0)-f'({\tilde x})=(x_0-{\tilde x})f''(\xi)\ge c(\alpha) {t}^{\frac1\alpha} {\tilde x}^{\alpha-2}\ge c(\alpha){t}^{1-\frac1\alpha}$, where $\xi\in(x_0,{\tilde x})$. Using~\eqref{eqn:x10} and $x_0\le c(\alpha){t}^{\frac1\alpha}$, it is rather straightforward to show that the numerator of~\eqref{eqn:x11} does not exceed $c(\alpha){t}$  leading to ${\bar x}\le c(\alpha){t}^{\frac1\alpha}$. Due to $x_0\le c(\alpha){t}^{\frac1\alpha}$ and~\eqref{eqn:x1x0-bound}, we have $-x_1<x_0-x_1\le c(\alpha){t}^{2-\frac1\alpha}\le  c(\alpha){t}^{\frac1\alpha}$, so, in summary, $x_3-x_1\le {\bar x} -x_1\le c(\alpha){t}^{\frac1\alpha}$, which is the upper bound on $x_3-x_1$ in~\eqref{eqn:l1-large-delta}.

	Now we prove that
	\begin{equation}\label{eqn:l2-large-delta-above}
	\text{if } {t} \ge c_2(\alpha) x_0^\alpha\text{ and } x_0\le c_1(\alpha)/2,\text{ then } x_1-x_2\approx c(\alpha) {t}^{\frac1\alpha}.
	\end{equation}
	Since $f$ is even and $l$ is decreasing, we have $x_2>-{\tilde x}$ (recall that ${\tilde x}>0$ is such that $f({\tilde x})=l(x_0)<l(x_1)$), so taking~\eqref{eqn:x1x0-bound} and~\eqref{eqn:x10} into account, we establish the upper bound on $x_1-x_2$ in~\eqref{eqn:l2-large-delta-above} as follows:
	\[
	x_1-x_2\le x_1+{\tilde x} \le x_0+{\tilde x} \le c(\alpha){t}^{\frac1\alpha}.
	\]
		
	Since $f$ is concave, we have $\{x:l(x)\le f(x)\}=[x_2,x_3]$. Therefore, to prove the lower bound on $x_1-x_2$ in~\eqref{eqn:l2-large-delta-above}, it is enough to show that there exists sufficiently small $c_7(\alpha)>0$ such that $l(x_1- c_7(\alpha){t}^{\frac1\alpha})<f(x_1- c_7(\alpha){t}^{\frac1\alpha})$, which would imply $x_1-x_2 \ge c_7(\alpha) {t}^{\frac1\alpha}$. If $c_7(\alpha)$ satisfies $\sqrt{2}c_7(\alpha) c_2(\alpha)^{\frac1\alpha-1}<\frac1{\sqrt{3}}-\frac12$, then by~\eqref{eqn:f'(x_0)-bounds} we get
	\begin{align*}
	l(x_1-c_7(\alpha){t}^{\frac1\alpha}) &= f'(x_0)(-c_7(\alpha){t}^{\frac1\alpha})+f(x_0)-\frac{{t}}{\sqrt{1+(f'(x_0))^2}} \\
	&< \sqrt{2} c_7(\alpha) x_0^{\alpha-1} {t}^{\frac1\alpha} + f(x_0) - \frac{{t}}{\sqrt{3}} \\ 
	&< \sqrt{2} c_7(\alpha) c_2(\alpha)^{\frac{1-\alpha}{\alpha}} {t} + f(x_0) - \frac{{t}}{\sqrt{3}} <f(x_0)-\frac{{t}}{2}.
	\end{align*}
	Next, if $x_1-c_7(\alpha){t}^{\frac1\alpha}\ge0$, then $f(x_0)-\frac{t}2<f(x_0)<f(x_1)<f(x_1-c_7(\alpha){t}^{\frac1\alpha})$  as $f$ is decreasing on $[0,1]$. Before proceeding, we note that by~\eqref{eqn:x1x0-bound} and ${t} \ge c_2(\alpha) x_0^\alpha$ we have $-x_1\le x_0-x_1\le \sqrt{2} c_2(\alpha)^{\frac1\alpha-1}t^{2-\frac1\alpha} \le c_7(\alpha) t^{\frac1\alpha}<\frac12$ if we assume that $c_0(\alpha)\le (c_7(\alpha) c_2(\alpha)^{1-\frac1\alpha}2^{-\frac12})^{\frac{\alpha}{2(\alpha-1)}}$ and $c_0(\alpha)\le (2c_7(\alpha))^{-\alpha}$. So, if $x_1-c_7(\alpha){t}^{\frac1\alpha}<0$, then by monotonicity of $f$ on $[-1,0]$, we see that $f(x_1-c_7(\alpha){t}^{\frac1\alpha})>f(-2c_7(\alpha){t}^{\frac1\alpha})=f(2c_7(\alpha){t}^{\frac1\alpha})$. Now we conclude as follows:
	\[
	f(x_0)-\frac{t}2\le 1 -\frac{{t}}{2} <1-\frac{2^{1-\frac1\alpha}2^\alpha c_7(\alpha)^\alpha}{\alpha}{t} <f(2c_7(\alpha){t}^{\frac1\alpha}),
	\]
	where $\frac{2^{1+\alpha-\frac1\alpha}c_7(\alpha)^\alpha}{\alpha}<\frac12$ was used for the second step, and~\eqref{eqn:f'(x_0)-behav} was used in the last step under the assumption that $c_0\le 2^{-1-\alpha}c_7(\alpha)^{-\alpha}$. We can choose $c_7(\alpha)>0$ arbitrarily to satisfy $c_7(\alpha)<\min\{c_2(\alpha)^{\frac1\alpha-1}(\frac1{\sqrt{6}}-\frac1{2\sqrt{2}}),\alpha^{\frac1\alpha}2^{\frac1{\alpha^2}-\frac2\alpha-1}\}$. Now~\eqref{eqn:l2-large-delta-above} is established.

	If $x_0=0$ and $c_0(\alpha)\le 1-2^{-\frac1\alpha}$, we can invoke~\eqref{eqn:f'(x_0)-behav} to immediately obtain that $l_i(t)=f^{-1}(1-t)\approx c(\alpha) t^{\frac1\alpha}$, $i=1,2$.
	
	Now we choose $c_0(\alpha)>0$ so that all the previously stated requirements (which were estimates from above on $c_0(\alpha)$) are fulfilled. The proof of~\eqref{eqn:li-comp} is complete as a combination of~\eqref{eqn:l1-small-delta}, \eqref{eqn:C2-case}, \eqref{eqn:l1-large-delta}, \eqref{eqn:l2-large-delta-above} if $x_0>0$ and the argument of the previous paragraph if $x_0=0$.		
\end{proof}

\end{document}